\documentclass[10pt]{article}

  \usepackage[english]{babel}
  \usepackage{amsthm}
  \usepackage{amsmath}
  \usepackage{amssymb}
  \usepackage{inputenc}
  \usepackage{setspace}
  \usepackage{enumerate}

  \usepackage{hyperref}

  \usepackage{dsfont}
  \usepackage[babel]{csquotes}
  \usepackage{geometry}
  \usepackage{bbm}
  \usepackage{stmaryrd}
  \usepackage[all]{xy}
  \usepackage{mathrsfs}

  \theoremstyle{plain}
  \newtheorem{thm}{Theorem}[section]
  \newtheorem{prop}[thm]{Proposition}
  \newtheorem{cor}[thm]{Corollary}
  \newtheorem{lem}[thm]{Lemma}
  \newtheorem{algo}[thm]{Algorithm}
  \theoremstyle{definition}
  \newtheorem{defn}[thm]{Definition}
  
  \newtheorem{rem}[thm]{Remark}

  \newcommand{\step}[1]{\par\medskip\par\noindent\textit{#1}} 

  \def \dS{\frac12\left(|\nabla\chi_i| +|\nabla\chi_j| - |\nabla(\chi_i+\chi_j)| \right)}
  \def \dSnull{\frac12\left(|\nabla\chi_i^0| +|\nabla\chi_j^0| - |\nabla(\chi_i^0+\chi_j^0)| \right)}
  \def \h{\sqrt{h}}
  \def \H{\mathscr{H}}

  \def \N{\mathbb{N}}
  \def \R{\mathbb{R}}

  \def \torus{{[0,\Lambda)^d}}

  \def \numphases{P}
  
  \def \chara{\mathbf{1}}
  
  \newcommand{\lauxbrakkedist}{ \mathrm{d}}
  \newcommand {\dist} {\mathop \mathrm{dist}}
  \DeclareMathOperator*{\esssup}{ess\,sup}

  \title{Brakke's inequality for the thresholding scheme}
 \author{Tim Laux\footnote{Department of Mathematics, University of California, Berkeley, CA 94720-3840 USA. Please use {tim.laux@math.berkeley.edu} for correspondence.} \and Felix Otto\footnote{Max-Planck-Institut f\"ur Mathematik in den Naturwissenschaften, Inselstra{\ss}e 22, 04103 Leipzig, Germany.}}
  \date{}

\begin{document}

    \maketitle
  \begin{abstract}
    We continue our analysis of the thresholding scheme from the variational viewpoint and prove a conditional convergence
    result towards Brakke's notion of mean curvature flow.
    Our proof is based on a localized version of the minimizing movements interpretation of Esedo\u{g}lu and the second author.
    We apply De Giorgi's variational interpolation to the thresholding scheme and pass to the limit in the resulting energy-dissipation inequality.
    The result is conditional in the sense that we assume the time-integrated energies of the approximations
    to converge to those of the limit.
    \medskip


    \medskip

    \end{abstract}

  \section{Introduction}
  
  The thresholding scheme is a time discretization for mean curvature flow.
  Its structural simplicity is intriguing to both applied and theoretical scientists.
  Merriman, Bence and Osher \cite{MBO92} introduced the algorithm in 1992 to overcome the numerical difficulty 
  of multiple scales in phase-field models.
  Their idea is based on an operator splitting for the Allen-Cahn equation, alternating between linear diffusion and thresholding.
  The latter replaces the fast reaction coming from the nonlinearity, i.e., the reaction-term, in the Allen-Cahn equation.
  We refer to Algorithm \ref{MBO} below for a precise description of the scheme in the multi-phase setting.
  The convolution can be implemented efficiently on a uniform grid using the Fast Fourier Transform and the thresholding step is a simple pointwise operation.
  Because of its simplicity and efficiency, thresholding received a lot of attention in the last decades.
  Large-scale simulations \cite{elsey2009diffusion,elsey2011large,elsey2011large2} demonstrate the efficiency of a slight modification of the scheme.
  For applications in materials science and image segmentation it is desirable to design algorithms that are efficient enough to handle large numbers of phases
  but flexible enough to incorporate external forces, grain-dependent and even anisotropic surface energies.
  Not long ago, the natural extension to the multi-phase case \cite{MBO94} was generalized to arbitrary surface tensions by Esedo\u{g}lu and the second author \cite{EseOtt14}. 
  In this paper, it was realized that thresholding preserves the gradient-flow structure of (multi-phase) mean-curvature flow
  in the sense that it can be viewed as a minimizing movements scheme for an energy that $\Gamma$-converges to the total interfacial area.
  This viewpoint allowed to incorporate a wide class of surface tensions including the well-known Read-Shockley formula for small-angle grain boundaries \cite{read1950dislocation}.
  
  \medskip
  
  The development of thresholding schemes for anisotropic motions started with the work \cite{ishii1999threshold} of Ishii, Pires and Souganidis.
  Efficient schemes were introduced by Bonnetier, Bretin and Chambolle \cite{bonnetier2012consistency}, where the convolution kernels are explicit and well-behaved in Fourier space but not necessarily in real space.
  The recent work \cite{elsey2016threshold} of Elsey and Esedo\u{g}lu is inspired by the variational viewpoint \cite{EseOtt14} and shows 
  that not all anisotropies can be obtained when structural features such as
  positivity of the kernel are required. However, variants of the scheme developed by Esedo\u{g}lu and Jacobs \cite{esedoglu2016convolution} share
  the same stability conditions even for more general kernels.  
  
  \medskip
  
  The rigorous asymptotic analysis of thresholding schemes started with the independent convergence proofs of
  Evans \cite{evans1993convergence} and Barles and Georgelin \cite{barles1995simple} in the isotropic two-phase case.
  Since the scheme preserves the geometric comparison principle of mean curvature flow,
  they were able to prove convergence towards the viscosity solution of mean curvature flow.
  Recently, Swartz and Yip \cite{SwaYip15} proved convergence for a smooth evolution by establishing consistency and stability of the scheme, very much in the flavor of classical numerical analysis.
  They prove explicit bounds on the curvature and injectivity radius of the approximations and get a good understanding of the transition layer.
  However, also their result does not generalize to the multi-phase case immediately.
  In our previous work \cite{laux2015convergence} we established the convergence of thresholding to a distributional formulation of multi-phase mean-curvature flow based on the assumption of convergence of the energies. In \cite{LauSwa15}, Swartz and the first author applied these techniques to the case of volume-preserving mean-curvature flow and other variants.
  
  \medskip

  Since the works \cite{barles1995simple, evans1993convergence} are based on the comparison principle, the proofs do not apply in the multi-phase case.
  Our guiding principle in this work is instead the gradient-flow structure of (multi-phase) mean curvature flow.
  In general, a gradient-flow structure is given by an energy functional and a metric tensor, which endows the configuration space with a Riemannian structure that encodes the dissipation mechanism.
  A simple computation reveals this structure for mean curvature flow.
  If the hypersurface $\Sigma = \Sigma(t)$ evolves smoothly by its mean curvature (here and throughout we use the time scale such that $2V=H$) the change of area is given by
  \begin{equation}\label{dt E}
   2\frac{d}{dt} |\Sigma| = -\int_\Sigma 2V \cdot H = - \int_{\Sigma} |H|^2,
  \end{equation}
  where $V$ denotes the normal velocity vector and $H$ denotes the mean curvature vector of $\Sigma$.
  Although \eqref{dt E} does not characterize the mean curvature flow one can read off the metric tensor, the $L^2$-metric $\int_\Sigma |V|^2$ on the space of normal vector fields, when fixing the energy to be the surface area.
  However, some care needs to be taken when dealing with this metric as for example the geodesic distance vanishes identically \cite{michor2004riemannian}.
  The implicit time discretization developed  by Almgren, Taylor and Wang \cite{ATW93} and Luckhaus and Sturzenhecker \cite{LucStu95} makes use of this gradient-flow structure.
  In fact, it inspired De Giorgi to define a similar implicit time discretization for abstract gradient flows which he named ``minimizing movements''.
  His abstract scheme consists of a family of minimization problems that mimic the principle of a gradient flow moving in direction of the steepest descent in an energy landscape.
  The configuration $\Sigma^n$ at time step $n$ is obtained from its predecessor $\Sigma^{n-1}$ by minimizing 
  $
   E(\Sigma) + \frac{1}{2h}\dist^2(\Sigma,\Sigma^{n-1}),
  $
  where $\dist$ denotes the geodesic distance induced by the Riemannian structure and $h>0$ denotes the time-step size. 
  In the case of a Euclidean configuration space, the scheme boils down to the implicit Euler scheme.
  In its Riemannian version, it has been used for applications in partial differential equations and for instance allowed Jordan, Kinderlehrer and the second author \cite{jordan1998variational}
  to interpret diffusion equations as gradient flows for the entropy w.r.t.\  the Wasserstein distance.  
  In view of the degeneracy in the case of mean curvature flow it is evident that the scheme in \cite{ATW93, LucStu95} uses a proxy for the geodesic distance.
The replacement for the distance of two boundaries $\Sigma=\partial \Omega $ and $ \tilde \Sigma =\partial  \tilde \Omega$ is the (non-symmetric) quantity 
  $4\int_{\Omega \Delta \tilde \Omega} d_{\tilde \Omega}\,dx $,
  where $d_{\tilde \Omega}$ denotes the (unsigned) distance to $ \partial \tilde\Omega$.
  Chambolle \cite{chambolle2007approximation} showed that the scheme \cite{ATW93, LucStu95} which seems academic at a first glance can be implemented rather efficiently. Recently, Bellettini and Kholmatov \cite{bellettini2017minimizing} analyzed the scheme in the multi-phase case. However, neither a conditional convergence result to a distributional BV-solution, nor one to a Brakke flow are available yet.
  
   \medskip
 
 Also Brakke's pioneering work \cite{brakke1978motion} is inspired by the gradient-flow structure of mean curvature flow.
  His definition is similar to the one of an abstract gradient flow and characterizes solutions by the optimal dissipation of energy.  
  Brakke measures the dissipation of energy only in terms of the mean curvature.
 As \eqref{dt E} cannot characterize the solution, Brakke monitors localized versions of the surface area, which leads to a sensible notion of solution; we refer to Definition \ref{def brakke flow bv} for a precise definition in our context of sets of finite perimeter.
  Ilmanen \cite{ilmanen1993convergence} used a phase-field version of Huisken's monotonicity formula \cite{huisken1990asymptotic}
  to prove the convergence of solutions to the scalar Allen-Cahn equation to Brakke's mean curvature flow.
  Extending his proof to the multi-phase case is a challenging open problem. 
  Only recently, Simon and the first author \cite{LauxSimon} proved a conditional convergence result for the vector-valued Allen-Cahn equation very much in the spirit of \cite{LucStu95, laux2015convergence}. However, an unconditional result is not yet available.
  Even the construction of non-trivial global solutions to multi-phase mean-curvature flow has only been achieved recently
  by Tonegawa and Kim \cite{kim2015mean}.
    \medskip
    
    In the present work we establish the convergence of the thresholding scheme to Brakke's motion by mean curvature. 
    As our previous result \cite{laux2015convergence}, also this one is only a conditional convergence result in the sense that we assume the 
    time-integrated energies to converge to those of the limit.
    Our proof is based on the observation that thresholding does not only have a \emph{global} minimizing movements interpretation, but indeed solves a \emph{family} of \emph{localized} minimization problems.
    In Section \ref{sec:brakke} we state our main results, in particular Theorem \ref{thm brakke inequality}.
    We use De Giorgi's variational interpolation for these localized minimization problems to derive an \emph{exact} 
    energy-dissipation relation and pass to the limit in the inequality with help of our strengthened convergence.
    We first recall the known results from the abstract framework of gradient flows in metric spaces (cf.\ Chapter 3 in \cite{ambrosio2006gradient}).
    Then we pass to the limit $h\to0$ in these terms with help of our strengthened convergence.
   It is worth pointing out that such a result is not known for the time discretization scheme \cite{ATW93, LucStu95}.
  \medskip
    
    The starting point for our analysis of thresholding schemes is the minimizing movements interpretation of Esedo\u{g}lu and the second author \cite{EseOtt14}. 
    Let us explain this interpretation with help of the example of the two-phase scheme. The combination  $\chi^n = \chara_{\{ G_h\ast \chi^{n-1} > \frac12\}}$ of convolution and thresholding is equivalent to 
    minimizing
    $ E_h(\chi) + \frac1{2h} \lauxbrakkedist_h^2(\chi,\chi^{n-1}),$
    where $E_h$ is an approximation of the perimeter functional and $\lauxbrakkedist_h$ is a metric.
    The latter serves as a proxy for the induced distance, just like $4 \int_{\Omega \Delta \Omega^{n-1}} d_{\Omega^{n-1}} dx$ in the minimizing movements scheme of Almgren, Taylor and Wang \cite{ATW93},
    and Luckhaus and Sturzenhecker \cite{LucStu95}.
    The $\Gamma$-convergence of similar functionals has been developed some time ago by Alberti and Bellettini \cite{alberti1998nonlocal} and more recently
    by Ambrosio, De Philippis and Martinazzi \cite{ambrosio2011gamma}, and was proven for the functionals $E_h$ by Miranda, Pallara, Paronetto and Preunkert \cite{miranda2007short}.
    Esedo\u{g}lu and the second author found a simpler proof in the case of the energies $E_h$, which extends to the multi-phase case.
    

  \medskip
    Let us recall the thresholding scheme and the basic notation.
    
   \begin{algo}\label{MBO}
  Given the partition $\Omega_1^{n-1},\dots,\Omega_P^{n-1}$ at time $t=(n-1)h$, 
  obtain the partition $\Omega_1^{n},\dots,\Omega_P^{n}$ at time $t=nh$ by the following two operations:
  \begin{enumerate}
  \vspace{-3pt}
  \item Convolution step:
  $
  \phi_i := G_h\ast \left( \sum_{j\neq i}\sigma_{ij} \chara_{\Omega_j^{n-1}} \right).
  $
    \item Thresholding step:
    $
    \Omega_i^n := \left\{ \phi_i < \phi_j \text{ for all } j\neq i\right\}.
    $
  \end{enumerate}
  \end{algo}
  
   Here and throughout the paper
  \begin{align*}
	  G_h(z) := \frac{1}{(2\pi h)^{d/2}} \exp\left(-\frac{|z|^2}{2h}\right)
  \end{align*}
  denotes the centered Gaussian of variance $h$, which we also think of as the heat kernel at time $\frac h2$.
  We assume the matrix of surface tensions $\sigma = (\sigma_{ij})_{i,j}$ to satisfy the obvious relations
  \[
  \sigma_{ij} = \sigma_{ji} >0 \; \text{for } i\neq j,\quad \sigma_{ii}=0
  \]
  and the usual (strict) triangle inequality
  \[
   \sigma_{ij} < \sigma_{ik}+\sigma_{kj}\quad \text{for all pairwise different } i,j,k.
   \]
   Furthermore, we  ask the matrix $\sigma$ to be conditionally negative definite
   \begin{equation}\label{sigma<0}
   \sigma <0 \quad \text{as a bilinear form on } \left(1,\dots,1\right)^\perp.
   \end{equation}
   This condition can be simply spelled out as $\xi \cdot \sigma \xi \leq -\underline{\sigma} |\xi|^2 $ for all $\xi \in \R^P$ such that $\sum_{i=1}^P \xi_i=0$,  where $\underline \sigma>0$ is a positive constant. The condition was introduced by Esedo\u{g}lu and the second author \cite{EseOtt14} and guarantees the dissipation of energy. Indeed, the conditional negativity \eqref{sigma<0} ensures that
   \[
   \left|\xi \right|_\sigma^2 := - \xi \cdot \sigma \xi = -\sum_{i,j} \sigma_{ij} \xi_i \xi_j, \quad \text{ for } \xi \in \R^P \text{ s.t.\ }\sum_i \xi_i=0
   \]
   defines a norm $\left| \mathop{\cdot} \right|_\sigma$ on the space $(1,\dots,1)^\perp$.
  For convenience we will work with periodic boundary conditions, i.e., on the flat torus $\torus$.
  We write $\int dx$ short for $\int_{\torus}dx$ and $\int dz$ short for $\int_{\R^d} dz$.
  Furthermore, $\chi^n$ given by $\chi_i^n:= \chara_{\Omega^n}$, $i=1,\dots,P$, denotes the vector of characteristic functions of the phases $\Omega_i^n$ at time step $n$ and we denote its piecewise constant interpolation by
  \[
   \chi^h(t) := \chi^n = \left( \chara_{\Omega_1^n},\dots, \chara_{\Omega_P^n} \right)\quad \text{for }t\in[nh,(n+1)h).
  \]
  However, we will mostly use a nonlinear interpolation which will be introduced later.
    Selim Esedo\u{g}lu and the second author \cite{EseOtt14} showed that thresholding preserves the gradient-flow structure of (multi-phase) mean curvature flow in the sense that it
    can be viewed as a minimizing movements scheme 
    \begin{equation}\label{MM interpretation}
    \chi^n = \arg \min_u \left\{ E_h(u) +  \frac1{2h} \lauxbrakkedist_h^2(u, \chi^{n-1})\right\},
    \end{equation}
    where the minimum runs over all measurable $u \colon \torus \to \R^P$ such that $0\leq u_i\leq 1$, $i=1,\dots,P$ and $\sum_i u_i = 1$ a.e..
    Here the dissipation functional
    \begin{equation}\label{def diss}
      \frac{1}{2h}\lauxbrakkedist^2_h(u,\chi) 
      := \frac1\h \int  \left|G_{h/2}\ast  \left( u-\chi\right)\right|_\sigma^2 dx
        =  -\frac1\h \int  G_{h/2}\ast  \left( u-\chi\right)
      \cdot \sigma\,G_{h/2}\ast  \left( u-\chi\right)dx
    \end{equation}
    is, because of \eqref{sigma<0}, the square of a metric
    and the energy
    \begin{equation}\label{def Eh}
     E_h(u) :=  \frac1\h \int u \cdot \sigma\, G_h\ast u \,dx 
    \end{equation}
    is an approximation of the total interfacial area.  Indeed,  this functional $\Gamma$-converges to the energy
    \[
     E(\chi) := c_0 \sum_{i,j} \sigma_{ij} \int \frac12 \left( \left|\nabla \chi_i\right| +\left|\nabla \chi_j\right| -\left|\nabla (\chi_i+\chi_j)\right| \right),
    \]
    defined for partitions $ \chi \colon \torus \to \{0,1\}^P$ s.t.\ $\sum_i \chi_i =1$.
    Writing $\Omega_i = \{\chi_i=1\}$ and $\partial^\ast \Omega_i$ for the reduced boundary of $\Omega_i$, the term
    \[
    \int \frac12 \left( \left|\nabla \chi_i\right| +\left|\nabla \chi_j\right| -\left|\nabla (\chi_i+\chi_j)\right| \right) = \H^{d-1}(\partial^\ast \Omega_i \cap \partial^\ast \Omega_j),
    \]
    is the measure of the interface between Phases $i$ and $j$,
    so that the energy $E$ is indeed the total interfacial area
    \[
     E(\chi) =c_0 \sum_{i,j} \sigma_{ij} \H^{d-1}(\partial^\ast \Omega_i \cap \partial^\ast \Omega_j).
    \]
    The constant $c_0$ is given by the first moment of $G$, i.e.,
     \[c_0 =  \int_0^\infty \left|z\right| G(z) \,dz= \frac1{\sqrt{2\pi}}.\] 
     The above mentioned $\Gamma$-convergence is an immediate consequence of the pointwise convergence of these functionals and the monotonicity property
    \begin{equation}\label{approximate monotonicity}
    	E_{N^2h}(u) \leq E_h(u) \quad \text{for all } u \colon \torus \to [0,1]^P, \text{ s.t.\ } \sum_{i=1}^P u_i =1, h>0, \text{ and } N\in \N,
    \end{equation}
    see \cite[Lemma A.2]{EseOtt14}.
    We write $A \lesssim B$ to express that $A\leq C B$ for
    a  generic constant $C<\infty$ that only depends on the dimension $d$, on the size $\Lambda$ of the domain, and the matrix $\sigma$ of surface tensions. By $A=O(B)$ we mean the quantitative $|A|\lesssim B$ while $A=o(B)$ as $h\to0$ means the qualitative $\frac A B \to0$ as $h\to0$.

    \section{Brakke's inequality and main result}\label{sec:brakke}
    
    The main statement of this work is Theorem \ref{thm brakke inequality} below. Assuming there was no drop of energy as $h\to0$, i.e.,
    \begin{align}\label{conv_ass}
    \int_0^T E_h(\chi^h)\, dt \to \int_0^T E(\chi)\,dt,
    \end{align}
    it states that the limit of the approximate solutions satisfies a $BV$-version of Brakke's inequality \cite{brakke1978motion}.
    
    \medskip
    
    Brakke's inequality is a weak formulation of motion by mean curvature $2V=H$ and is motivated by the following characterization of the normal velocity.
    Given a smoothly evolving hypersurface $\partial \Omega(t)=\Sigma(t)$ with normal velocity vector $V$ we have
    \begin{align}\label{brakke inequality smooth2}
      \frac{d}{dt} \int_{\Sigma} \zeta  = \int_\Sigma \left( -\zeta\, H \cdot V + V \cdot \nabla \zeta  + \partial_t \zeta\right)
    \end{align}
    for any smooth test function $\zeta \geq 0$. 
    The converse is also true: Given a function $V\colon \Sigma \to \R$ such that \eqref{brakke inequality smooth2} holds for any such test function $\zeta\geq 0$
    then $V$ is the normal velocity of $\Sigma$.
    In the pioneering work \cite{brakke1978motion}, Brakke uses this idea for his definition of  the equation $2V=-H$ 
    to extend the concept of motion by mean curvature to general varifolds.
    We recall his definition in our more restrictive setting of finite perimeter sets, which in the smooth two-phase case simplifies to the inequality
	\begin{align}\label{brakke inequality smooth1}
       2 \frac{d}{dt} \int_{\Sigma} \zeta  \leq \int_\Sigma \left( -\zeta \, |H|^2 +H\cdot \nabla \zeta + 2\partial_t \zeta\right).
    \end{align}
   
    \begin{defn}\label{def brakke flow bv}
    We say that the time-dependent partition $\chi\colon (0,T) \times \torus \to \{0,1\}^P \in BV$ with $\sum_i \chi_i =1$ a.e.\ moves by mean curvature with initial data $\chi^0 \colon \torus \to \{0,1\}^P \in BV$ with $\sum_i \chi^0_i =1 $ a.e.\ if there exists a $\sum_{i,j}\sigma_{ij} \dS dt$-measurable normal vector field
    $H\colon (0,T)\times \torus \to \R^d$ with
    \[
     \sum_{i,j} \sigma_{ij}\int_0^T\int |H|^2 \dS dt< \infty,
    \]
    which is the mean curvature vector of the partition in the sense that for all test vector fields $\xi \in C_c^\infty((0,T)\times \torus,\R^d)$
    \begin{equation}\label{def H}
    \begin{split}
	 \sum_{i,j} \sigma_{ij}\int_0^T \int \left(\nabla \cdot \xi - \nu_i \cdot \nabla \xi \, \nu_i \right) &\dS dt  \\
	 =  -\sum_{i,j} \sigma_{ij}&\int_0^T\int H\cdot \xi \,\dS dt,
	 \end{split}
    \end{equation} 
    such that for any test function 
    $\zeta\in C^\infty([0,T]\times \torus)$ with $\zeta \geq 0$ and $\zeta(\,\cdot\,,T)=0$ we have
    \begin{equation}\label{brakke inequality int}
      \begin{split}
      \sum_{i,j} \sigma_{ij}\int_0^T \int \left( - \zeta \, |H|^2 +  H\cdot \nabla \zeta   + 2\partial_t \zeta \right)& \dS dt\\
     \geq - \sum_{i,j} \sigma_{ij}  \int &\zeta(\,\cdot\,,0)\dSnull
    \end{split}
    \end{equation}
    Here and throughout, $\nu_i$ denotes the measure theoretic normal of Phase $i$ characterized by the equation $\nabla \chi_i = \nu_i \left|\nabla\chi_i\right|$.
    Note that with this choice, $\nu_i$ points inwards.
    \end{defn}

    
    Equation \eqref{def H} encodes not only that $H$ is the mean curvature vector along the smooth part of the surface cluster but furthermore enforces the Herring angle condition along triple junctions, which comes from the integration by parts rule for smooth hypersurfaces $\Sigma$ with boundary $\Gamma$:
    \[
     \int_\Sigma \left(\nabla \cdot \xi -\nu \cdot \nabla \xi \, \nu\right) = \int_\Gamma \xi \cdot b
    -  \int_\Sigma H\cdot \xi,
    \]
    where $b$ denotes its conormal.
    
	Equation \eqref{brakke inequality int} does not only encode the mean curvature flow equation via the optimal dissipation of energy but also the initial data $\chi^0$ in a weak sense.
    
    \begin{thm}[Brakke's inequality]\label{thm brakke inequality}
    Given initial data $\chi^0 \colon \torus \to \{0,1\}^P$ with $E(\chi^0)< \infty$ and a finite time horizon $T<\infty$, for any sequence there exists a subsequence $h\downarrow0$ such that 
    the approximate solutions given by Algorithm \ref{MBO} 
    converge to a limit $\chi\colon (0,T)\times \torus \to \{0,1\}^P$ in $L^1$ and a.e.\ in space-time. Given the convergence assumption (\ref{conv_ass}),
    $\chi$ evolves by mean curvature in the sense of Definition \ref{def brakke flow bv}.
    \end{thm}
    
    \begin{rem}\label{rem compactness}
      Given initial conditions $\chi^0$ with $E(\chi^0)<\infty$ 
      the compactness in \cite[Proposition 2.1]{laux2015convergence} yields a subsequence such that $\chi^h\to \chi$ in $L^1$ and a.e.\ for a partition $\chi$ with $\esssup_t E(\chi(t)) \leq E(\chi^0)$.
    \end{rem}
   
    This statement is similar to our result in \cite{laux2015convergence}.
    There we proved the convergence of thresholding towards a distributional formulation of (multi-phase) mean-curvature flow, the same notion as in \cite{LucStu95}.
    Under the same assumption \eqref{conv_ass} as in the present work, for any $i=1,\dots,P$ we constructed a $|\nabla\chi_i|\,dt$-measurable function $\tilde V_i\colon (0,T) \times \torus \to \R$ with
    \[
      \int_0^T\int \tilde V_i^2 \left| \nabla \chi_i \right| dt < \infty,
    \]
    which is the (scalar) normal velocity of the $i$-th phase in the sense that
    \[
    \int_0^T \int \partial_t \zeta\, \chi_i \, dx\,dt  + \int \zeta(\,\cdot\, 0) \chi_i^0 \,dx= - \int_0^T \int \zeta \, \tilde V_i \left| \nabla \chi_i\right| dt
    \]
    for all $\zeta \in C^\infty([0,T]\times \torus)$ with $\zeta(\,\cdot\,,T)=0$,
    such that
    \begin{align}\label{v=H}
    \sum_{i,j} \sigma_{ij}\int_0^T \int \left( \nabla \cdot \xi -\nu_i \cdot \nabla \xi \, \nu_i - 2\, \xi \cdot \nu_i \, \tilde V_i \right) \frac12\left(|\nabla\chi_i| +|\nabla\chi_j| - |\nabla(\chi_i+\chi_j)| \right) dt=0
    \end{align}
    for all $ \xi \in C_0^\infty((0,T)\times \torus,\R^d)$.
    
    Without any regularity assumption, none of the two formulations is stronger in the sense that it implies the other.
    Nevertheless (\ref{v=H}) requires more regularity as it is formulated for sets of finite perimeter, whereas
    Brakke's inequality naturally extends to general varifolds.
	Finally, we note that our proof here is much softer than the result \eqref{v=H} proved in our earlier work \cite{laux2015convergence}.

    \section{De Giorgi's variational interpolation and idea of proof}\label{de_giorgi}
    
    
    It is a well-appreciated fact that a classical gradient flow $\dot u(t) = - \nabla E(u(t))$ of a smooth energy functional $E$ 
    on a Riemannian manifold can be characterized by the optimal rate of dissipation of the energy $E$ along the solution $u$:
    \begin{equation}\label{gradient flow}
    \frac{d}{dt} E(u(t)) \leq -\frac12 |\dot u(t)|^2 - \frac12 |\nabla E(u(t))|^2.
    \end{equation}
    This is the guiding principle in generalizing gradient flows to metric spaces where one replaces $|\dot u|$ by the metric derivative and $|\nabla E(u)|$ by some upper gradient, 
    e.g.\ the \emph{local slope} $|\partial E(u)|$, see \eqref{def local slope} for a definition in our context.
    
    As discussed in the introduction, mean curvature flow can be viewed as a gradient flow in the sense that for a smooth evolution $\Sigma=\Sigma(t)$ the energy, which in this case is the surface area $|\Sigma(t)|$, satisfies
    the inequality
    \[
    2 \frac{d}{dt} |\Sigma| = -\int_\Sigma   2V \cdot H \leq -\frac12 \int_\Sigma |H|^2  - \frac12 \int_\Sigma |2V|^2.
    \]
    While in the abstract framework, the dissipation of the energy is measured w.r.t.\ both terms $|\dot u|^2 \hat = \int_\Sigma |2V|^2$ and $|\partial E(u)|^2 \hat = \int_\Sigma |H|^2$,
    Brakke measures the rate only in terms of the local slope $\int_\Sigma H^2$ but asks for the \emph{localized} version \eqref{brakke inequality int}, which due to \eqref{brakke inequality smooth2} in the simple setting of a single surface evolving smoothly by its mean curvature is precisely \eqref{brakke inequality smooth1} with equality.
    
    \medskip
    
    The basis of this work is the approximate version of Brakke's inequality, Lemma \ref{lem MM interpretation} below.
    In view of the minimizing movements interpretation \eqref{MM interpretation} it should be feasible to obtain at least the \emph{global}
    inequality
    \[
     2 \frac{d}{dt} |\Sigma| \leq -\int_\Sigma |H|^2
    \]
    but the localized inequality \eqref{brakke inequality int} would be still out of reach.
    The lemma states that thresholding does not only solve the \emph{global} minimization problem \eqref{MM interpretation}
    but a whole \emph{family} of \emph{local} minimization problems, which will allow us to establish the family of localized inequalities \eqref{brakke inequality int}.
    
    \begin{lem}[Local minimality]\label{lem MM interpretation}
      Let $\chi^n$ be obtained from $\chi^{n-1}$ by one iteration of Algorithm \ref{MBO} and $\zeta\geq0$ an arbitrary test function. Then
      \begin{align}\label{E F d2}
       \chi^n = \arg \min_u \left\{ E_h(u,\chi^{n-1};\zeta) + \frac1{2h} \lauxbrakkedist_h^2(u,\chi^{n-1};\zeta)\right\},
      \end{align}
      where the minimum runs over all $u\colon \torus \to [0,1]^P$ with $\sum_i u_i =1$ a.e.. 
      By $\lauxbrakkedist_h(u,\chi;\zeta)$ we denote the localization of the metric $\lauxbrakkedist_h(u,\chi)$ given by
      \begin{align}\label{def local dist}
      \frac1{2h} \lauxbrakkedist_h^2(u,\chi;\zeta) &:= \frac1\h \int \zeta \left| G_{h/2} \ast\left( u-\chi\right) \right|^2_\sigma dx,
      \end{align}
      which is again a (semi-)metric on the space of all such $u$'s as above and in particular satisfies a triangle inequality.
      By $E_h(u,\chi;\zeta)$ we denote the localized (approximate) energy incorporating the localization error in both energy and metric:
      \begin{equation}
      \begin{split}
      E_h(u,\chi;\zeta) := \frac1\h \int \zeta\,u \cdot \sigma \,G_h\ast u\, dx
	&+ \frac1\h \int   \left(u-\chi\right) \cdot \sigma \left[\zeta,G_h\ast\right] \chi\,  dx \label{def local E}\\
		&- \frac1\h \int   \left(u-\chi\right) \cdot \sigma \left[\zeta,G_{h/2}\ast\right] G_{h/2}\ast \left(u- \chi\right)   dx.
	 \end{split}
      \end{equation}
    \end{lem}
    Here and throughout the paper
    \begin{equation*}
      \left[\zeta,G_h\ast\right] u := \zeta \, G_h\ast u - G_h\ast \left( \zeta\,u\right) \approx -\nabla \zeta \cdot h\nabla G_h\ast u
    \end{equation*}
    denotes the commutator
    of the multiplication with the (smooth) function $\zeta$ and the convolution with the kernel $G_h$, both of which act componentwise on the vector $u$.
  
    \medskip
    
    Let us briefly comment on the structure of the localized energy $E_h$.
    First, by definition of $E_h$ we have
	\begin{align*}
	  E_h(u,u;\zeta) = \frac1\h \int \zeta\, u \cdot \sigma\, G_h\ast u\,dx \quad \text{and}\quad 	  E_h(u,\chi;1) =E_h(u), \quad \text{cf.\ } \eqref{def Eh},
	\end{align*}
    so that in particular we recover the minimizing movements interpretation \eqref{MM interpretation} in the case $\zeta\equiv 1$.
    Second, with the localization $\zeta$, the first integral in the definition of $E_h$ is an approximation of the \emph{localized} total interfacial energy $c_0 \sum_{i,j} \sigma_{ij}\int_{\Sigma_{ij}} \zeta$.
    We will see shortly that the second term gives rise to the transport term in Brakke's inequality, while the last term, the commutator arising from the metric term, will be shown to be negligible in the limit $h\to0$ for all quantities under consideration here.
      
    \medskip
    
    Thanks to the \emph{local} minimization property \eqref{E F d2} of the thresholding scheme
    we can apply the abstract framework of De Giorgi, cf.\ Chapters 1--3 in \cite{ambrosio2006gradient}, to this localized setting.
    As for any minimizing movements scheme, the comparison of $\chi^n$ to the previous time step $\chi^{n-1}$ in the minimization problem \eqref{E F d2} yields an energy-dissipation inequality which 
    serves well as an a priori estimate, but which fails to be sharp by a factor of $2$.
    To obtain a \emph{sharp} inequality we follow the ideas of De Giorgi. We introduce his \emph{variational interpolation} $u^h$ of $\chi^n$ and $\chi^{n-1}$:
    For $t\in (0,h]$ and $n\in \N$ we let 
    \begin{equation}\label{def var intpol}
     u^h( (n-1)h + t) := \arg \min_u \left\{ E_h(u,\chi^{n-1};\zeta) + \frac1{2t} \lauxbrakkedist_h^2(u,\chi^{n-1};\zeta)\right\}.
    \end{equation}
   Note that the choice of $u^h$ is not necessarily unique, but given $u^h(nh)=\chi^n$, we can choose $u^h$ to depend continuously on $t$ w.r.t.\ the metric $d_h$.
    Comparing $u^h(t)$ with $u^h(t+\delta t)$ in this minimization problem and taking the limit $\delta t  \to 0$ while keeping $h$ fixed, one obtains the sharp energy-dissipation inequality along this interpolation,
    the following approximate version of Brakke's inequality \eqref{brakke inequality int}.
    
	It is worth pointing out that opposed to the special case $t=h$, we do not have an explicit formula for the interpolations $u^h$, and there is no guarantee for $u^h_i \in \{0,1\}$;
	indeed we expect that generically $u^h_i \in (0,1)$.
	 
    \begin{cor}[Approximate Brakke inequality]\label{cor brakke inequality discr}
    For any test function $\zeta\geq 0 $, a time-step size $h>0$ and $T=Nh$ we have 
    \begin{align}
      &\frac h2 \sum_{n=1}^N \big| \partial  E_h(\,\cdot\,, \chi^{n-1};\zeta) \big|^2(\chi^n)
      + \frac12 \int_0^{T} \big| \partial  E_h(\,\cdot\,, \chi^h(t);\zeta) \big|^2(u^h(t))\, dt\notag\\
      & \qquad \qquad+ h\sum_{n=1}^N \frac1h \left( E_h(\chi^n,\chi^{n-1};\zeta)-E_h(\chi^n,\chi^n;\zeta) \right)
      \leq E_h(\chi^0,\chi^0;\zeta) -E_h(\chi^N,\chi^N;\zeta),\label{brakke inequality discr}
    \end{align}
    where $ \left| \partial E_h(\,\cdot\,,\chi;\zeta ) \right|(u)$ is the ``local slope" of $E_h(\,\cdot\,,\chi;\zeta) $ at $u$ defined by
    \begin{align}\label{def local slope}
      \left| \partial E_h(\,\cdot\,,\chi;\zeta ) \right|(u) := \limsup_{v\to u} \frac{\left( E_h(u,\chi;\zeta) - E_h(v,\chi;\zeta)\right)_+}{\lauxbrakkedist_h(u,v;\zeta)}.
    \end{align}
    The convergence $v\to u$ is in the sense of the metric $\lauxbrakkedist_h$.
    \end{cor}
    
    Our goal is to derive Brakke's inequality \eqref{brakke inequality int} from its approximate version \eqref{brakke inequality discr}, i.e., we want to relate the limits of the
    expressions in \eqref{brakke inequality discr} to the terms appearing in \eqref{brakke inequality int}: In Propositions \ref{prop sum F} we will show that the transport term arises from the increments 
    $\frac1h \left( E_h(\chi^n,\chi^{n-1};\zeta)-E_h(\chi^n,\chi^n;\zeta) \right)$. Then we will derive a $\liminf$-inequality between the local slope of the approximate energies $E_h$ and the squared mean curvature of the limiting partition in Proposition \ref{prop H and dE}.

    \medskip

    

     \medskip
     
    %

    While Corollary \ref{cor brakke inequality discr} is a mere application of the abstract theory in \cite{ambrosio2006gradient}, we will now use the particular character of thresholding, i.e.,
    the structure of the energy \eqref{def local E} and the metric term \eqref{def local dist} in order to pass to the limit in the approximate Brakke inequality \eqref{brakke inequality discr}.
    
    We start with the basic a priori estimate for the piecewise constant interpolation $\chi^h$.
    \begin{cor}[Energy-dissipation estimate]\label{cor energy dissipation estimate}
      Given initial conditions $\chi^0\colon \torus \to \{0,1\}^\numphases$ with finite energy $E_0:=E(\chi^0)<\infty$, a time-step size $h>0$ and a finite time horizon $T=Nh$ we have
      \begin{align}
      & \sup_N \left(E_h(\chi^N) +  h\sum_{n=1}^N\frac{\lauxbrakkedist_h^2(\chi^n,\chi^{n-1})}{2h^2} \right) \leq E_0. \label{energy dissipation estimate}
      \end{align}
      \end{cor}
        
    We recall the following proposition from \cite{laux2015convergence} which will allow us to pass to the limit
    in the approximate Brakke inequality for the scheme. It is only for this proposition we use the convergence assumption \eqref{conv_ass}.
  
    \begin{prop}[Lemma 2.8 and Proposition 3.5 in \cite{laux2015convergence}]\label{prop conv G K}
      Given $u^h\to \chi$ and $E_h(u^h) \to E(\chi)$, for any test function $\zeta\in C^\infty(\torus)$ it holds
      \begin{equation}
	   \frac1\h \int \zeta \,u^h \cdot \sigma \, G_h \ast u^h \, dx \to c_0 \sum_{i,j} \sigma_{ij} \int \zeta \frac12\left(|\nabla\chi_i| +|\nabla\chi_j| - |\nabla(\chi_i+\chi_j)| \right),\label{conv G(z)}
	 \end{equation}
	 and for any test matrix field $A \in C^\infty(\torus,\R^{d\times d})$ we have
	 \begin{equation}
	 \sum_{i,j} \sigma_{ij} \frac1\h \int A \colon  u_i^h\, h\nabla^2 G_h \ast u_j^h \, dx  \to c_0 \sum_{i,j} \sigma_{ij} \int  \nu_i \cdot A  \nu_i \frac12\left(|\nabla\chi_i| +|\nabla\chi_j| - |\nabla(\chi_i+\chi_j)| \right) .\label{conv D2G(z)}
      \end{equation}
    \end{prop}

  In \cite{laux2015convergence} we used the above proposition to pass to the limit in the first variation of the energy
  \begin{equation}\label{def delta E}
   \delta E_h(u,\xi) := \frac{d}{ds}\Big|_{s=0} E_h(u_s),
  \end{equation}
  where the inner variations $u_s$ of $u$ along a vector field $\xi$ are given by the transport equation
  \begin{align}\label{def us}
    \partial_s u_s + \left(\xi \cdot \nabla\right) u_s = 0 \quad u_s|_{s=0}=u.
  \end{align}
  \begin{prop}[Proposition 3.2, Remark 3.3 and Lemma 3.4 in \cite{laux2015convergence}]\label{prop dE and dE old}
  Given $u\colon \torus \to [0,1]^\numphases$ with $\sum_i u_i =1$ a.e.\ we have
  \begin{equation}\label{dE and dE old}
  \Big| \delta E_h(u,\xi) 
   - \frac1\h \sum_{i,j} \sigma_{ij}\int \nabla \xi \colon  u_i \left( G_h \, Id - h \nabla^2 G_h \right) \ast u_j \,dx\Big| \lesssim \h \|\nabla^2 \xi \|_\infty E_h(u).
   \end{equation}
  In particular if $u^h\to \chi\in \{0,1\}^P$ and $E_h(u^h)\to E(\chi)<\infty$ we have
  \begin{equation}\label{old dE to dE}
   \delta E_h(u^h,\xi) \to  c_0 \sum_{i,j} \sigma_{ij} \int \nabla \xi \colon \left( Id-\nu_i\otimes\nu_i\right) \dS.
  \end{equation}
  \end{prop}
  
  \begin{rem}
  Although the proof is contained in \cite{laux2015convergence}, we will recall the short argument for \eqref{dE and dE old} and rephrase it in the language of commutators to introduce the reader to the notation. 
  Note that the argument for \eqref{old dE to dE} is only based on the fact that the measure on the right-hands side of \eqref{conv G(z)} agrees with the trace of the right-hand side measure of \eqref{conv D2G(z)}.
  \end{rem}

  
  \medskip
  
  In the absence of the localization, i.e., if the test function $\zeta$ is constant, the last left-hand side term 
  $
    h\sum_{n=1}^N \frac1h\left( E_h(\chi^n,\chi^{n-1};\zeta)-E_h(\chi^n,\chi^n;\zeta) \right)
  $
   in \eqref{brakke inequality discr} vanishes. 
   However, for a non-constant test function we have to pass to the limit in this extra term.
   Let us again restrict ourselves to the two-phase case for the following short discussion to see that formally, the behavior of this term is obvious.
   Expanding $\zeta$ (and ignoring the commutator in the metric term for a moment), the leading-order term of the increments on the left-hand side of \eqref{brakke inequality discr} as $h\to 0$ is 
    \begin{equation}\label{heuristics}
    \begin{split}
      \frac1h\left(E_h(\chi^n,\chi^{n-1};\zeta)-E_h(\chi^n,\chi^{n};\zeta)\right)
     &\approx \int   \frac{\chi^n-\chi^{n-1}}{h}  \frac1\h\left[\zeta,G_h\ast\right] (1-\chi^{n-1})\,  dx\\
     &\approx \int   \frac{\chi^n-\chi^{n-1}}{h} \,\nabla\zeta \cdot \h\nabla G_h \ast \chi^{n-1}\,  dx,
     \end{split}
    \end{equation}
    which at least formally (and after integration in time) converges to $  -c_0\int_0^T\int_\Sigma  V \cdot \nabla \zeta$. Hence we expect to recover the
    transport term $-\frac{c_0}2\int_0^T\int_\Sigma H \cdot \nabla \zeta$ in Brakke's inequality \eqref{brakke inequality int} by using the equation $2V=H$ once.
  In the following proposition we make this step rigorous under the convergence assumption \eqref{conv_ass}. 
	More precisely, we prove the estimate
	\[
	\left| \frac1h\left(E_h(u,\chi;\zeta)-E_h(u,u;\zeta)\right)
	- \delta \Big( \frac1{2h} \lauxbrakkedist_h^2(\,\cdot\,,\chi)\Big) (u,\nabla \zeta) \right|
	\lesssim_{\zeta,\xi}   \sum_{p=1}^2 \Big(h^{1/4} \frac{\lauxbrakkedist_h(u,\chi)}{h}\Big)^p
	+ h^{1/4} E_h(\chi),
	\]  
	where the implicit constant in $\lesssim_{\zeta,\xi}$ depends on the test fields $\zeta$ and $\xi$. Afterwards we apply the Euler-Lagrange equation 
	$
	\delta \big( \frac1{2h} \lauxbrakkedist_h^2(\,\cdot\,,\chi)\big) (u,\xi)
	=-\delta E_h(u,\xi)
	$
	of the global minimizing movements principle \eqref{MM interpretation} and Proposition \ref{prop dE and dE old} to obtain the transport term in the form 
  \[
   \frac{c_0}2 \int_0^T \int \nabla^2 \zeta  \colon \left( Id-\nu\otimes \nu\right) \left|\nabla \chi\right|dt.
   \]

  \begin{prop}\label{prop sum F}
  Given the convergence assumption \eqref{conv_ass} and with $T=Nh$ we have
  \begin{align*}
    \lim_{h\to 0} h\sum_{n=1}^N& \frac1h\left( E_h(\chi^n,\chi^{n-1};\zeta) - E_h(\chi^n,\chi^n;\zeta) \right) \\
    &= \frac {c_0}2 \sum_{i,j} \sigma_{ij} \int_0^T \int \nabla^2 \zeta  \colon \left( Id-\nu_i\otimes \nu_i\right)\dS dt.
  \end{align*}
  \end{prop}

    The following a priori estimate for the variational interpolation $u^h$ defined in \eqref{def var intpol} follows now very easily.
    
      \begin{cor}[A priori estimate]\label{cor a priori estimate}
      Given initial conditions $\chi^0\colon \torus \to \{0,1\}^\numphases$ with finite energy $E_0:=E(\chi^0)<\infty$, a time-step size $h>0$ and a finite time horizon $T=Nh$, if the test function $\zeta$ is strictly positive, then for the interpolation \eqref{def var intpol} we have
      \begin{align}\label{energy dissipation estimate interpolation}
       \limsup_{h\downarrow0} \Big( \sup_t E_h(u^h(t)) +  \int_0^T \frac{\lauxbrakkedist_h^2(u^h(t),\chi^{h}(t))}{2 h^2} dt \Big)
       <\infty.
      \end{align}
      In particular, we have the following \emph{quantitative} proximity of $u^h(t)$ to $\chi^h(t)$ in our metric:
  \begin{equation}\label{quantitative closeness}
   \h \int_0^T \int \Big| G_{h/2} \ast \Big( \frac{u^h-\chi^h}{h}\Big) \Big|_\sigma^2 dx\,dt \quad \text{stays bounded as} \quad h\to0.
  \end{equation}
    \end{cor}  
  
 The following statement is a post-processed version of our assumption \eqref{conv_ass}.
  \begin{lem}\label{la conv ass for var int}
  Given the convergence assumption (\ref{conv_ass}), for a subsequence, we also have the pointwise in time property
  \begin{align}\label{conv_ass pointwise}
    E_h(\chi^h) \to E(\chi) \quad \text{a.e. in } (0,T)
  \end{align}
  and furthermore for the variational interpolation $u^h$ given by \eqref{def var intpol}
  \begin{align}\label{conv_ass var int pointwise}
    E_h(u^h) \to E(\chi) \quad \text{a.e. in } (0,T).
  \end{align}
  Moreover, Lebesgue's dominated convergence theorem implies the integrated version
  \begin{align}\label{conv_ass var int}
    \int_0^T E_h(u^h)\,dt \to \int_0^T E(\chi)\,dt.
  \end{align}
  Additionally, the interpolations $u^h$ converge to the same limit in $L^1$, i.e., $\lim_h u^h = \lim_h \chi^h =\chi$.
  \end{lem}
  
  \medskip
 
   In the following proposition, we probe the definition of the local slope \eqref{def local slope} with inner variations $u_s$. These are given by the transport equation \eqref{def us} and in the simpler two-phase case we obtain
    \begin{equation}\label{iop slope and dE}
    \big| \partial  E_h(\,\cdot\,, \chi;\zeta) \big|(u) \geq \frac{\delta E_h(\,\cdot\,,\chi;\zeta)(u,\xi)}{\sqrt{2 \h \int \zeta \left( G_{h/2}\ast (\xi \cdot \nabla u) \right)^2 dx}}.
    \end{equation}
     Then we will find that the localization $\zeta$ acts trivially on this term: As $h\to0$, the first variation of the localized energy  $\delta E_h(\,\cdot\,,\chi;\zeta)(u,\xi)$ behaves like the first variation of the global energy in direction of the localized vector field $\zeta\xi$, i.e., $ \delta E_h(\,\cdot\,)(u,\zeta \xi)$:
     \[
    \big| \delta E_h(\,\cdot\,,\chi;\zeta)(u,\xi)-\delta E_h(\,\cdot\,)(u,\zeta\xi) \big|
    \lesssim_{\zeta,\xi} h^{1/4} \frac{\lauxbrakkedist_h(u,\chi)}{h}.
     \]
    Similarly, it is straight-forward to see that 
    \[
    \bigg| \h \int \zeta \left( G_{h/2} \ast (\xi \cdot \nabla u) \right)^2 dx-
    \frac1\h \int \zeta \, \xi \otimes \xi \colon (1-u) h\nabla^2 G_h \ast u\,dx \bigg| \lesssim_{\zeta,\xi}  h^{1/4} E_h(u),
    \]
    where again the implicit constant depends on $\zeta $ and $\xi$.
    Taking the limit $h\to0$, we may apply Proposition \ref{prop conv G K} to both terms, the numerator and the denominator.
    Then taking the supremum over all possible vector fields $\xi$ we obtain the $\liminf$-inequality
    \[
    \frac{c_0}{2} \int_0^T \int \zeta H^2 \left| \nabla \chi \right| dt \leq \liminf_{h\to0}\frac12 \int_0^T\big| \partial  E_h(\,\cdot\,, \chi^h;\zeta) \big|^2(u^h)\,dt,
    \]
    which provides the final ingredient for the proof of Theorem \ref{thm brakke inequality}.
	
  \begin{prop}\label{prop H and dE}
  Let $\zeta>0$ be smooth, $\chi^h(t)$ the approximate solution obtained by Algorithm \ref{MBO} and let $u^h(t)$ be either the variational interpolation \eqref{def var intpol} or the approximate solution $\chi^h(t+h)$
  at time $t+h$. 
  Given the convergence assumption \eqref{conv_ass}, there exists a measurable normal vector field $H\in L^2(\sum_{i,j} \sigma_{ij}\dS dt) $ which is the mean curvature vector of the partition $\chi$ in the sense of \eqref{def H}, such that
  \begin{align}\label{H and dE}
  \frac{c_0}2 \sum_{i,j} \sigma_{ij} \int_0^T \int \zeta \,|H|^2 \dS  dt \leq \liminf_{h\to 0} \int_0^T \left| \partial E_h(\,\cdot \, , \chi^h;\zeta)\right|^2(u^h) \, dt.
  \end{align}
  \end{prop}

    \section{Proofs}\label{sec:proofs}
       
    We first give the proofs of the main results, Theorem \ref{thm brakke inequality}, Lemma \ref{lem MM interpretation} and Corollary \ref{cor brakke inequality discr} and then turn to the other statements which form the basis of the proof of Theorem \ref{thm brakke inequality}.
    
  \begin{proof}[Proof of Theorem \ref{thm brakke inequality}]

  \step{Step 1: Time-freezing for $\zeta$.}
  We claim that it is enough to prove
  \begin{align}
      \sum_{i,j} \sigma_{ij}\int_0^{\tilde T} \int& \left( \zeta\, |H|^2 - H\cdot \nabla \zeta \right) \dS dt\notag\\
      \leq&2 \sum_{i,j} \sigma_{ij}\int \zeta  \dS\Big|_{t=0} \notag \\
   &     -2 \sum_{i,j} \sigma_{ij} \int \zeta  \dS \Big|_{t=\tilde T}\label{brakke inequality const test}
  \end{align}
  for any time-independent, strictly positive test function $\zeta=\zeta(x)>0$ and a.e.\ $\tilde T$.
  
This is a standard approximation argument:  In order to reduce (\ref{brakke inequality int}) to (\ref{brakke inequality const test})
  we fix a time-dependent test function $\zeta=\zeta(t,x)\geq 0$ and two time instances $0\leq s< t$.
  It is no restriction to assume $s=0$. Writing $t=:\tilde T$ for the time horizon we take a regular partition $0=T_0< \dots < T_M = \tilde T$ of the interval $(0,\tilde T)$ of fineness $\tau=\tilde T/M$.
  We write $\zeta_M$ for the piecewise constant interpolation of $\zeta$ plus a small perturbation $\frac1M$ so that $\zeta_M \geq \frac1M >0$:
  \begin{align*}
    \zeta_M(t) := \zeta(T_{m-1}) + \frac1M\quad \text{if } t\in [T_{m-1},T_m).
  \end{align*}
  Writing $\partial^{-\tau}\zeta_M (t) := \frac1\tau \left( \zeta_M(t) - \zeta_M(t-\tau)\right)$ for the discrete (backwards) time derivative we have
  \begin{align}\label{etaM to eta}
    \zeta_M \to \zeta,\quad   \nabla \zeta_M \to \nabla \zeta \quad \text{and} \quad \partial^{-\tau} \zeta_M \to \partial_t \zeta \quad \text{uniformly as } M\to \infty.
  \end{align}
  Using \eqref{brakke inequality const test} for $\zeta_M\geq \frac1M >0$ on each interval $[T_{m-1},T_m)$ and summing over $m$ we obtain \eqref{brakke inequality int}.
  \step{Step 2: Proof of (\ref{brakke inequality const test}).}
  Given a test function $\zeta=\zeta(x)>0$, we want to prove (\ref{brakke inequality const test}) for a.e.\ $\tilde T>0$.
  For simplicity, we may assume that $\tilde T=Nh$ is a multiple of the time step size $h$.
  Furthermore by \eqref{conv_ass pointwise} we may assume that
  $E_h(\chi^h(\tilde T))\to E(\chi(\tilde T))$.
  We pass to the limit in the approximate Brakke inequality (\ref{brakke inequality discr}) to prove Brakke's inequality \eqref{brakke inequality const test} for this time-independent test function.
  
  By \eqref{conv_ass} we may apply Proposition \ref{prop H and dE} to obtain
  \begin{align*}
    \frac{c_0}4 \sum_{i,j} \sigma_{ij} \int_0^{\tilde T} \int \zeta  \,&|H|^2 \dS  dt\\
    &\leq \liminf_{h\to 0} \frac h2 \sum_{n=1}^N \left| \partial E_h(\,\cdot\,, \chi^{n-1};\zeta) \right|^2(\chi^n),
  \end{align*}
  as well as
  \begin{align*}
    \frac{c_0}4 \sum_{i,j} \sigma_{ij} \int_0^{\tilde T} \int \zeta \,&|H|^2 \dS dt\\
    &\leq \liminf_{h\to 0}  \frac12 \int_0^{\tilde T} \left| \partial E_h(\,\cdot\,, \chi^h(t);\zeta) \right|^2(u^h(t))\, dt.
  \end{align*}
  In addition, we may apply Proposition \ref{prop sum F} for the transport term and after division by the common prefactor $c_0$ we obtain (\ref{brakke inequality const test}).
  \end{proof}
    
  \begin{proof}[Proof of Lemma \ref{lem MM interpretation}]
     Given initial conditions $\chi\in\{0,1\}^\numphases$ with $\sum_i \chi_i =1$ and a time-step size $h>0$, one iteration of the thresholding scheme yields $\chi^1_i = \chara_{\{(\sigma G_h\ast \chi)_i = \min_j (\sigma G_h\ast \chi)_j\}}.$ 
    Then $\chi^1$ clearly minimizes
    \begin{align*}
    2 u \cdot \sigma \, G_h\ast \chi
    \end{align*}
    among all $u\in [0,1]^P$ s.t.\ $\sum_{i} u_i =1$.
    This expression is equal to
    \begin{align*}
   u &\cdot \sigma \, G_h\ast u 
   - (u-\chi) \cdot \sigma \,G_h\ast (u -\chi)  
    +  
		u\cdot \sigma \, G_h \ast \chi - \chi \cdot \sigma \, G_h \ast u   +\chi \cdot \sigma \,G_h\ast \chi\\
    &=  u \cdot \sigma \, G_h\ast u 
   - (u-\chi) \cdot \sigma \,G_h\ast (u -\chi)  +   (u-\chi)\cdot \sigma \, G_h \ast \chi - \chi \cdot \sigma \, G_h \ast (u-\chi)   +\chi \cdot \sigma \,G_h\ast \chi,
    \end{align*}
    where the last right-hand side term is independent of $u$ and thus irrelevant for the minimization.
    Multiplying with $\zeta\geq 0$ and integrating shows that $\chi^1$ minimizes 
    \[
     \int \!\zeta \left[u\cdot \sigma  \,G_h\ast u - \left(u - \chi\right)\cdot \sigma \, G_h\ast \left(u -\chi\right) +(u-\chi) \cdot \sigma \, G_h\ast \chi
      -\chi \cdot \sigma \, G_h\ast \left(u-\chi\right) \right] dx + \text{const.}
    \]
    Dividing by $\h$, recalling the definitions \eqref{def local dist} and \eqref{def local E} of the localized distance and energy, and using the semi-group and symmetry properties of the kernel and the symmetry of $\sigma$ yield \eqref{E F d2}.  
    \end{proof}

Corollary \ref{cor brakke inequality discr} is an immediate consequence of interpreting our problem from the point of view of gradient flows in metric spaces.
    
    Given $\chi$ and $\zeta$, the \emph{Moreau-Yosida approximation} $E_{h,t}$ of $E_h$ is defined by
      \begin{align*}
      E_{h,t}(\chi;\zeta) :=  \min_{u} \left\{ E_h(u,\chi;\zeta) + \frac1{2t} \lauxbrakkedist_h^2(u,\chi;\zeta)\right\}
      \end{align*}
      and furthermore we recall the (not necessarily unique) \emph{variational interpolation} $u^h(t)$ of $\chi$ and $\chi^1:=u^h(h)$, cf.\ \eqref{def var intpol}.
    
      As $t$ decreases we have a stronger penalization and thus we expect $u^h(t)$ to be ``closer'' to $\chi=u^h(0)$ than $\chi^1=u^h(h)$ which justifies the name ``interpolation''.
      Note that $E_h(u,\chi;\zeta)$ and $\lauxbrakkedist_h(u,\chi;\zeta)$ are, because of the smoothing property of the kernel $G_h$, weakly continuous in $u$ and $\chi$.
      Furthermore, we recall that we choose $u^h(\,\cdot\,)$ in such a way that it is continuous in $t$ w.r.t.\ the metric $\lauxbrakkedist_h$.
      
      The following general theorem monitors the evolution of the (approximate) energy along the interpolation $u^h(t)$ in terms of the
    distances at different time instances measured by the metric $\lauxbrakkedist_h$, and gives a
    lower bound in terms of the local slope $|\partial E_h|$ of $E_h$, cf.\ \eqref{def local slope}.
    
       Because of the localization, our energy \eqref{def local E} depends on the configuration at the previous time step.
    However, we can apply the abstract framework (cf.\ Chapter 3 of \cite{ambrosio2006gradient}) to this case if we only follow one time step.
    Both $h$ and $\zeta$ are fixed parameters when applying these results.

    \begin{thm}[Theorem 3.1.4 and Lemma 3.1.3 in \cite{ambrosio2006gradient}]\label{thm dE/dt}
    For every $\chi \colon \torus \to \{0,1\}^\numphases$ with $\sum_i \chi_i =1$ a.e.\ the map $t\mapsto E_{h,t}(\chi;\zeta)$ is locally Lipschitz in $(0,h]$ and continuous in $[0,h]$ with
    \begin{align}\label{energy equality}
     &\frac t2 \left|\partial E_h(\,\cdot\,,\chi;\zeta)\right|^2 (u^h(t)) + \frac12 \int_0^t \left|\partial E_h(\,\cdot\,,\chi;\zeta)\right|^2(u^h(s))\, ds \notag\\
     &\qquad \leq
     \frac1{2t} \lauxbrakkedist_h^2(u^h(t),\chi;\zeta) + \int_0^t \frac{\lauxbrakkedist_h^2(u^h(s),\chi;\zeta)}{2 s^2} ds  =
     E_h(\chi,\chi;\zeta) - E_h(u^h(t),\chi;\zeta). 
    \end{align}
    \end{thm}
    
    The idea behind Theorem \ref{thm dE/dt} is rather simple: By testing the minimality of $u^h(t)$ against $u^h(s)$ and taking $s\uparrow t$ (and similarly with reversed roles for $s\downarrow t$) one obtains $\frac d{dt} E_{h,t}(\chi;\zeta) = - \frac{\lauxbrakkedist_h^2(u^h(t),\chi;\zeta)}{2 t^2}$. Integrating this equation from $t=0$ to $t=h$ yields the equality between the energy difference and the the metric term in \eqref{energy equality}.
    The first inequality between the local slope and the metric term comes from the general estimate  $\left|\partial E_h(\,\cdot\,,\chi;\zeta)\right|(u^h(t)) \leq \frac{\lauxbrakkedist_h(u^h(t),\chi;\zeta)}{t}$, which follows from the definition of the local slope and the triangle inequality.
    
  \begin{proof}[Proof of Corollary \ref{cor brakke inequality discr}]
     We apply \eqref{energy equality} in Theorem \ref{thm dE/dt} with $\chi=\chi^{n-1}$ and $t=h$, and sum over $n=1,\dots,N$.
  \end{proof}
  
  Now we turn to the more problem-specific statements, which use the special character of thresholding.
  
 \begin{proof}[Proof of Corollary \ref{cor energy dissipation estimate}]
	The statement simply follows from testing the global minimization problem \eqref{MM 				interpretation} for $\chi^n$ with its predecessor $\chi^{n-1}$ and summation over $n$.
 \end{proof}

  \begin{proof}[Proof of Proposition \ref{prop dE and dE old}]
  The first variation of $E_h$ at $u$ along the vector field $\xi$ defined through \eqref{def delta E} and \eqref{def us} is given by
  \begin{align}
    \delta E_h(u,\xi) 
    = &\frac1\h \int -\left(\xi\cdot\nabla\right) u \cdot \sigma \,G_h\ast u
  -u \cdot \sigma \, G_h \ast \left(\left(\xi\cdot\nabla\right) u\right)dx  \label{dE without any tricks}\\
  = & \frac1\h \int u \cdot \sigma \left(\xi\cdot\nabla\right) G_h \ast u 
  - u\cdot \sigma  \left(\nabla   G_h \ast \left(\xi \otimes u\right) \right)dx	\notag\\
  &+ \frac1\h \int \left( \nabla\cdot \xi\right) u \cdot \sigma \,  G_h \ast  u
  + u\cdot \sigma \,  G_h \ast \left(\left(\nabla\cdot\xi\right) u\right)dx.\notag
  \end{align}
  This can be compactly rewritten as
  \[
      \delta E_h(u,\xi) = \frac1\h \int 2 \left(\nabla \cdot \xi \right) u\cdot \sigma \, G_h\ast u +  u \cdot \sigma \left[\xi \cdot,\nabla G_h\ast \right] u 
      - u \cdot \sigma \left[ \nabla \cdot \xi, G_h\ast\right] u \,dx.
  \]
  Componentwise in $u$, we expand the first commutator:
  \begin{align*}
  \left(\left[ \xi \cdot, \nabla G_h\ast \right] u_i\right) (x)  &= \int \left( \xi(x)-\xi(x-z)\right) \cdot \nabla G_h(z) \,u_i(x-z)\,dz\\
  &= \nabla \xi(x) \colon \int - \frac{z}{\h}\otimes \frac{z}{\h} G_h(z)\, u_i(x-z)\,dz
  + O\left( \|\nabla^2 \xi \|_\infty \big(\h \,k_h\ast u_i \big)(x) \right),
  \end{align*}
  where we used the identity $\nabla G(z)=-G(z)z$ and where the kernel $k_h$ is given by the mask $k(z)=|z|^3 G(z)$ and can be controlled by a Gaussian with slightly larger variance $k(z) \lesssim G(z/2)$.
  Likewise, the second commutator can be estimated pointwise by 
  \[
  \left|\left[ \nabla \cdot \xi, G_h\ast\right]u _i\right| \lesssim \|\nabla^2\xi\|_\infty\h \,\tilde k_h\ast  u_i,
  \]
   where $\tilde k_h$ is given by the mask $\tilde k(z)=|z|\, G(z)\lesssim G(z/2)$.
  By the identity $G(z)\left( Id-z\otimes z\right)=-\nabla^2 G(z) $ we indeed obtain \eqref{dE and dE old} with an error of order $\|\nabla^2 \xi\|_\infty \h \, E_{4h}(u)$, which by the monotonicity \eqref{approximate monotonicity} of $E_h$ yields the claim.
  \end{proof}
  
  \begin{proof}[Proof of Proposition \ref{prop sum F}]
  We first note that by definition \eqref{def local E},
  \begin{align*}
    E_h(\chi^n,\chi^{n-1};\zeta) - E_h(\chi^n,\chi^n;\zeta) =&\frac1\h \int   \left(\chi^n-\chi^{n-1}\right) \cdot \sigma \left[\zeta,G_h\ast\right] \chi^{n-1} dx \\
	  & - \frac1\h \int   \left(\chi^n-\chi^{n-1}\right) \cdot \sigma \left[\zeta,G_{h/2}\ast\right] G_{h/2}\ast \left(\chi^n- \chi^{n-1}\right)   dx.
  \end{align*}
  By the antisymmetry of the commutator (and the symmetry of $\sigma$), we may replace $\chi^{n-1}$ by $\chi^n$ on the right-hand side:
  \[
  \frac1\h \int   \left(\chi^n-\chi^{n-1}\right)\cdot \sigma  \left[\zeta,G_h\ast\right]  \chi^{n}  - \left(\chi^n-\chi^{n-1}\right) \cdot \sigma  \left[\zeta,G_{h/2}\ast\right] G_{h/2}\ast \left(\chi^n- \chi^{n-1}\right)   dx.
  \]
  Now we prove the proposition in two steps.
  First, we show that the first term converges to the right-hand side of the claim:
    \begin{equation}\label{claim transport term}
    \begin{split}
     \lim_{h\to0} \int_0^T \int   &\partial_t^{-h} \chi^h \cdot \sigma \, \frac1\h\left[\zeta,  G_h\ast\right] \chi^{h}  dx\\
      &=  
	    c_0 \sum_{i,j} \sigma_{ij} \int_0^T \int  \nabla^2 \zeta\colon \left( Id-\nu_i\otimes\nu_i \right)\dS dt,
	    \end{split}
    \end{equation}
  where $\partial_t^{-h} \chi^h = \frac{\chi^h -\chi^h(\,\cdot\,-h)}{h}$ denotes the discrete backwards time derivative of $\chi^h$.
  Then we prove that the second term is negligible:
  \begin{align}\label{error transport term}
      \lim_{h\to0}\int_0^T \h \int   \partial_t^{-h} \chi^h \cdot \sigma  \left[\zeta,G_{h/2}\ast\right] G_{h/2}\ast \partial_t^{-h} \chi^h \,   dx\, dt =0.
  \end{align}
  
  \step{Step 1: Argument for \eqref{claim transport term}.}  
    Expanding the commutator to second order
    \begin{equation}
      \frac1\h\left[\zeta, G_h\ast\right]v =   \h \nabla G_h \ast \left(- \nabla \zeta \,v\right) + \frac\h2 \left( G_h\, Id + h\nabla^2 G_h\right) \ast \left( \nabla^2 \zeta\, v\right)
      +  O\left(\|\nabla^3\zeta\|_\infty h \, k_h \ast |v|\right), \label{commutator estimate G}
    \end{equation}
    where the kernel $k_h$ is given by the mask $k(z)=|z|^3 G(z)$,
    we obtain for the first-order term
    \begin{equation*}
    \begin{split}
    &h\sum_{n=1}^N  \int   \frac{\chi^n-\chi^{n-1}}{h}  \cdot \sigma  \h \nabla G_h\ast \left(  -\nabla \zeta\,\chi^{n} \right)\, dx\\
    &\qquad \qquad \qquad=  h\sum_{n=1}^N \frac1\h \int \left(\chi^n-\chi^{n-1} \right) \cdot \sigma \,G_h \ast\left( -\left(\nabla \zeta \cdot \nabla \right) \chi^n -\Delta \zeta \left(1-\chi^{n}\right)\right) dx.
    \end{split}
    \end{equation*}
    Now we recognize the first variation of the (unlocalized) dissipation functional, cf.\ \eqref{def diss}, on the right-hand side:
    \[
    \delta \left( \frac1{2h} \lauxbrakkedist_h^2(\,\cdot\,,\chi^{n-1}) \right)(\chi^{n},\xi) = -\frac2\h \int \left(\chi^n-\chi^{n-1} \right) \cdot \sigma \, G_h \ast \left(- \left(\xi \cdot \nabla \right)\chi^n \right) dx
    \]
    with $\nabla \zeta $ playing the role of $\xi$.
    Using the semi-group and symmetry properties of the kernel, the extra term involving the Laplacian of the test function can be estimated by Jensen's inequality and the energy-dissipation estimate \eqref{energy dissipation estimate}:
    \begin{align} \label{d2zeta}
		&\left|h\sum_{n=1}^N \frac1\h \int \left(\chi^n-\chi^{n-1} \right)\cdot \sigma \, G_h \ast \left(\Delta \zeta \,\chi^{n}\right)dx\right| \\
		&\qquad \quad \lesssim \|\Delta \zeta\|_\infty \left( \frac T\h h\sum_{n=1}^N \frac1\h\int \left|G_{h/2}\ast\left(\chi^n-\chi^{n-1}\right) \right|_\sigma ^2dx\right)^{1/2} 
		\stackrel{\eqref{def diss},\eqref{energy dissipation estimate}}{\leq} \|\Delta \zeta\|_\infty T^{1/2} E_0^{1/2} h^{1/4}.\notag
    \end{align}
    Formally, the leading-order term, i.e., the first variation of the dissipation functional, converges to the transport term, which in the two-phase case is $-c_0 \int_\Sigma V\cdot \nabla \zeta$. Since instead we want to obtain
    the term $-\frac{c_0}2 \int_\Sigma H\cdot,\nabla \zeta $ (in its weak form $\frac{c_0}2 \int_\Sigma \nabla^2 \zeta \colon (Id-\nu\otimes \nu)$), we employ the minimizing movements interpretation \eqref{MM interpretation}
    in form of its Euler-Lagrange equation
    \begin{align*}
    \delta E_h(\chi^n,\xi) + \delta \left( \frac1{2h} \lauxbrakkedist_h^2(\,\cdot\,,\chi^{n-1}) \right)(\chi^{n},\xi) = 0 \quad \text{for all } \xi \in C^\infty(\torus,\R^d).
    \end{align*}
    We thus have
    \[
    h\sum_{n=1}^N  \int   \frac{\chi^n-\chi^{n-1}}{h}  \cdot \sigma \h \nabla G_h\ast \left(  \nabla \zeta \cdot \nabla \chi^{n} \right)\, dx
    =
    \frac h2\sum_{n=1}^N \delta E_h(\chi^{n},\nabla \zeta).
    \]
    By the convergence of the energies \eqref{conv_ass} we may apply Proposition \ref{prop conv G K} in \eqref{dE and dE old} and pass to the limit $h\to0$ in the right-hand side:
    \[
    \begin{split}
     \lim_{h\downarrow0}\frac12& \int_0^T \delta E_h(\chi^{h},\nabla \zeta)\,dt\\
    & = \frac{c_0}2 \sum_{i,j} \sigma_{ij}\int_0^T \int \nabla^2 \zeta \colon \left( Id - \nu_i \otimes \nu_i\right)\dS dt.
     \end{split}
    \]
    Indeed, by Lebesgue's dominated convergence theorem, we are allowed to interchange the order of integration in time and the limit $h\to0$.
    
    Now we conclude the argument for \eqref{claim transport term} by showing that the contributions of the second- and third-order terms in the expansion \eqref{commutator estimate G} are negligible in the limit $h\to0$.
    The contribution of the second-order term is estimated as follows. We argue componentwise in $\chi^h$, fix $i,j\in \{1,\dots, P\}$,  $i\neq j$ and observe that by Cauchy-Schwarz
    \begin{align*}
      & \int_0^T  \int   \partial_t^{-h} \chi^h_i  \frac\h2 \left( G_h\, Id + h\nabla^2 G_h \right)\ast \left(  \nabla^2 \zeta \,\chi_j^{h}\right)\, dx\, dt\\
      &\leq \left( \int_0^T  \h \int \left| \left( G_h\, Id + h\nabla^2 G_h \right)\ast \partial_t^{-h} \chi_i^h \right|^2dx \,dt\right)^{\!\!\frac12}\!
      \left( \int_0^T  \h \int \left| \nabla^2 \zeta \chi_j^h\right|^2dx \,dt\right)^{\!\!\frac12}\!.
    \end{align*}
    The second right-hand side integral is bounded by $T \Lambda^d \|\nabla^2 \zeta\|_\infty^2 \h \to0$, while the first right-hand side integral can be estimated by
    \[
      \h \int \left| \left( G_h\, Id + h\nabla^2 G_h \right)\ast \partial_t^{-h} \chi_i^h \right|^2dx \lesssim 
      \h \int \left( G_{h/2}\ast \partial_t^{-h} \chi_i^h \right)^2 dx \stackrel{\eqref{sigma<0},\eqref{energy dissipation estimate}}{\lesssim}  E_h(u^0),
    \]
    where in the first estimate we have used the semi-group property $G_h\, Id + h\nabla^2 G_h =( G_{h/2}\, Id + h\nabla^2 G_{h/2} )\ast G_{h/2}$ and the fact that the kernel $G_{h/2}\, Id + h\nabla^2 G_{h/2}$ is uniformly bounded in $L^1$.
    
    Since the kernel $k_h$ is uniformly bounded in $L^1$, the contribution of the third-order term in \eqref{commutator estimate G} is controlled by
    \[
    \int_0^T \int h \left| \partial_t^{-h} \chi^h \right| dx\,dt = \int_0^T \int \left|\chi^h(t)-\chi^h(t-h)\right|dx\,dt.
    \]
    The following basic estimate, which is valid for any pair of characteristic functions,
    \begin{equation}\label{chi-chitilde}
     \left|\chi-\tilde \chi\right| =  \left|\chi-\tilde \chi\right| ^2 \lesssim 
     \left|G_{h/2}\ast \left(\chi-\tilde \chi\right) \right|^2 + \left|G_{h/2}\ast \chi - \chi\right|^2 + \left|G_{h/2}\ast\tilde \chi -\tilde \chi\right|^2,
    \end{equation}
    and the fact that by the normalization $\int G_{h/2}(z)\,dz=1$ and the pointwise estimate $G_{h/2}(z) \lesssim G_h(z)$, for each $i\in \{1,\dots,\numphases\}$, we have
    \begin{align}
		\frac1\h \int \left| G_{h/2}\ast \chi_i -\chi_i\right| dx
		&\leq \frac1\h  \int G_{h/2}(z) \int \left|\chi_i(x)-\chi_i(x-z)\right| dx\,dz \notag\\
		&\lesssim \frac1\h \int G_{h}(z) \int \left|\chi_i(x)-\chi_i(x-z)\right|dx\,dz \notag\\
		&= \frac1\h \int G_{h}(z) \int (1-\chi_i)(x)\chi_i(x-z) + (1-\chi_i)(x-z) \chi_i(x) \,dx\,dz\notag\\
		&= \frac2\h \sum_{1\leq j \leq \numphases, j\neq i} \int \chi_i G_h\ast\chi_j \,dx \leq \frac{1}{\min_{i\neq j} \sigma_{ij}} E_h(\chi)\notag
    \end{align}
    yield the estimate
    \[
     \int_0^T \int \left|\chi^h(t)-\chi^h(t-h)\right|dx\,dt \lesssim \left(1+T\right) E_0 \h\to0.
    \]
    This concludes the proof of \eqref{claim transport term}.
    
    \step{Step 2: Argument for \eqref{error transport term}.}
    We may argue componentwise and omit the index in the following. We expand the commutator to first order
    \begin{equation}
    \left[ \zeta,G_{h/2}\ast\right] v =  \nabla G_{h/2} \ast \left(-\frac{h}{2} \nabla \zeta \, v\right) + O\left( \|\nabla^2 \zeta\|_\infty  h\, k_h\ast |v|\right),
     \label{commutator estimate G1}
    \end{equation}
    where the kernel $k_h$ is given by the mask $k(z)=|z|^2 G_{1/2}(z)$,
    and first consider the contribution of the first-order term to \eqref{error transport term}, namely
    \[
      -\frac h2 \int_0^T \h \int \partial_t^{-h} \chi^h \nabla G_{h/2} \ast \left( \nabla \zeta \, G_{h/2} \ast \partial_t^{-h} \chi^h \right) dx\,dt.
    \]
    Using the antisymmetry of $\nabla G$, the chain rule and integration by parts this is equal to
    \begin{align*}
    &\frac h2 \int_0^T \h \int \nabla \left(G_{h/2} \ast \partial_t^{-h}\chi^h \right) \cdot \nabla \zeta \left( G_{h/2} \ast \partial_t^{-h} \chi^h \right)dx\,dt\\
    &\qquad= \frac h2 \int_0^T \h \int  \nabla \zeta \cdot \nabla \left( \frac12 \left(G_{h/2} \ast \partial_t^{-h} \chi^h \right)^2\right)dx\,dt\\
    &\qquad= -\frac h4 \int_0^T \h \int  \Delta \zeta  \left( G_{h/2} \ast \partial_t^{-h} \chi^h \right)^2 dx\,dt.
    \end{align*}
    By the energy-dissipation estimate \eqref{energy dissipation estimate} this term vanishes as $h\to0$.
    
    The contribution of the second-order term coming from the expansion \eqref{commutator estimate G1} is controlled by
    \begin{align*}
      \int_0^T \h \int \left|\partial_t^{-h} \chi^h \right| h\, k_h\ast \left| G_{h/2}\ast \partial_t^{-h} \chi^h \right| dx \,dt
    \lesssim \int_0^T \h \int   \left|G_{h/2}\ast \partial_t^{-h} \chi^h \right| dx \,dt.
    \end{align*} 
    Therefore, this term  vanishes as $h\to0$ by Jensen's inequality and the energy-dissipation estimate \eqref{energy dissipation estimate}.
  \end{proof}
  
 \begin{proof}[Proof of Corollary \ref{cor a priori estimate}] 
 	In contrast to the piecewise constant interpolation $\chi^h$, the variational interpolation 		$u^h$ is not given in an explicit form but only by the minimization problem \eqref{def var intpol}. In particular, since in general $u^h$ may depend on the test function $\zeta$, we are tied to the local minimization problem \eqref{def var intpol}.
 	By \eqref{energy equality} we have in particular 
 	 \begin{align*}
 	 	 E_h(u^h(\tilde T),u^h(\tilde T);\zeta)+ \int_0^{\tilde T} \frac{\lauxbrakkedist_h^2(u^h,\chi^h;\zeta)}{2h^2}dt 
 	 	\leq E_h(\chi^0,\chi^0;\zeta)- \sum_{n=1}^N \left( E_h(\chi^n,\chi^{n-1};\zeta)-E_h(\chi^n,\chi^n;\zeta) \right)
 	 \end{align*}
 	 for any $\tilde T\in[Nh,(N+1)h)$, where $N \in \N$.
 	 The left-hand side is bounded from below by
 	 \[
		\inf \zeta  \left(E_h(u^h(\tilde T)) + \int_0^{\tilde T}  \frac{\lauxbrakkedist_h^2(u^h,\chi^h)}{2h^2}dt \right)
 	 \]
 	 while the right-hand side can be controlled by Proposition \ref{prop sum F}.
 \end{proof}
   
 \begin{proof}[Proof of Lemma \ref{la conv ass for var int}]
  
  The convergence assumption \eqref{conv_ass} together with the $\liminf$-inequality of the $\Gamma$-convergence implies the convergence $E_h(\chi^h)\to E(\chi)$ in $L^1(0,T)$.
  In order to understand the behavior of the energies of the variational interpolations $u^h$ we compare them to the energies of the piecewise constant interpolation:
  \begin{align*}
  	|E_h(u^h)-E_h(\chi^h)  | 
  	&= \frac1\h \left| \int \left( G_{h/2} \ast u^h\cdot \sigma\, G_{h/2}\ast ( u^h-\chi^h) + 
  	G_{h/2} \ast (u^h-\chi^h) \cdot \sigma \, G_{h/2}\ast \chi^h\right) dx \right|\\
  	&\lesssim \frac1\h \int \left| G_{h/2}\ast (u^h-\chi^h) \right|_\sigma dx
  \end{align*}
  and by Jensen we obtain
  \[
  	\int_0^T|E_h(u^h)-E_h(\chi^h)  |\,dt \lesssim  T^{1/2} \frac1\h \left(\int_0^T \int \left|G_{h/2}\ast (u^h-\chi^h)\right|_\sigma^2dx\,dt \right)^\frac12,
  \]
  which by \eqref{quantitative closeness} vanishes as $h\to0$.
  That means the approximate energies converge to the same limit in $L^1(0,T)$
   and therefore we obtain the $L^1$-convergence \eqref{conv_ass var int} and -- after the possible passage to a further subsequence -- the pointwise convergences \eqref{conv_ass pointwise} and \eqref{conv_ass var int pointwise}.
   
   The convergence of $u^h$ to $\chi=\lim \chi^h$ can now be proven by the very same argument as the one following \eqref{chi-chitilde}. The only difference here is that the components $u^h_i$ are not characteristic functions. 
   Then the first equality in \eqref{chi-chitilde} can be replaced by the inequality $|u-\chi| \leq 2 |u-\chi|$ and in the chain of inequalities following \eqref{chi-chitilde}, the equality $|\chi_i(x)-\chi_i(x-z)|=(1-\chi_i)(x)\chi_i(x-z) + (1-\chi_i)(x-z) \chi_i(x)$ can simply be replaced by the inequality $|u_i(x)-u_i(x-z)| \leq (1-u_i)(x)u_i(x-z) + (1-u_i)(x-z) u_i(x)$, which is valid for any function $u_i $ with values in $[0,1]$. Together with \eqref{energy dissipation estimate interpolation}, this leads to the estimate
   \[
	\limsup_{h\to0} \frac1\h \int_0^T \int \left| u^h-\chi^h\right| dx\,dt  <\infty,
   \]
   so that indeed $u^h\to \chi$ in $L^1$.
  \end{proof}
  
 \begin{proof}[Proof of Proposition \ref{prop H and dE}]
  
  We give ourselves a test vector field $\xi$ and let the variations $u_s$ defined in \eqref{def us} play the role of $v$ in the definition of the local slope \eqref{def local slope} so that we obtain the inequality
  \begin{align*}
    |\partial E_h(\,\cdot\,,\chi^h;\zeta)|(u^h) \geq \limsup_{s\to 0} \frac{\left(E_h(u^h,\chi^h;\zeta)-E_h(u^h_s,\chi^h;\zeta)\right)_+}{\lauxbrakkedist_h(u^h_s,u^h;\zeta)}.
  \end{align*}
  As $s\to 0$ we expand the numerator in the following way
  \begin{align*}
    E_h(u^h_s,\chi^h;\zeta) = E_h(u^h,\chi^h;\zeta) + s \,\frac{d}{ds}\Big|_{s=0} E_h(u^h_s,\chi^h;\zeta) +o(s) \quad \text{as } s\to 0.
  \end{align*}
  For the denominator we have, cf.\ \eqref{def diss},
  \begin{align*}
    \frac1{2h}\lauxbrakkedist_h^2(u^h_s,u^h;\zeta) = \frac{s^2}\h  \int \zeta \left| G_{h/2}\ast\left(\left(\xi \cdot \nabla\right) u^h\right)\right|_\sigma^2 dx + o(s^2) \quad \text{as } s\to 0.
  \end{align*}
  Taking the limit $s\to 0$ we obtain
  \begin{equation}\label{lower bound dE}
    |\partial E_h(\,\cdot\,,\chi^h;\zeta)|(u^h) \geq \frac{\frac{d}{ds}\big|_{s=0} E_h(u^h_s,\chi^h;\zeta)}{\sqrt{2\h  \int \zeta \left| G_{h/2}\ast\left(\left(\xi \cdot \nabla\right) u^h\right)\right|_\sigma^2 dx}}.
  \end{equation}
Now we expand $\zeta$ and $\xi$  to analyze the leading order terms as $h\to0$.
Using \eqref{def us} we compute the first variation of the localized energy $E_h(u,\chi;\zeta)$, cf.\ \eqref{def local E}:
  \begin{align*}
      \frac{d}{ds}\Big|_{s=0}E_h& (u_s,\chi;\zeta)\\
      = \frac1\h \! \int \!\!&- \zeta \left(\xi \cdot \nabla\right) u \cdot \sigma \, G_h\ast u - \zeta \, u\cdot \sigma \, G_h \ast \left( \left(\xi \cdot \nabla\right) u\right) \\
      &-  \left(\xi \cdot \nabla\right) u \cdot \sigma \left[ \zeta, G_h\ast\right] u 
      + \left(\xi \cdot \nabla\right) u  \cdot \sigma \left[ \zeta,G_h\ast\right] \left( u-\chi \right)\\
      &+ \left(\xi \cdot \nabla\right) u \left[ \zeta, G_{h/2}\ast\right] \! G_{h/2} \ast ( u-\chi) - \left( \xi \cdot \nabla \right)u \cdot \sigma\,G_{h/2} \ast \!\left[ \zeta, G_{h/2}\ast\right]\! \left(u-\chi\right) dx.
  \end{align*}
  The fourth term in the sum comes from replacing $\chi$ by $u$ in the third term, while for the last term we used the antisymmetry 
  $\int u \left[ \zeta, G_{h/2}\ast\right]v \,dx= - \int v \left[ \zeta, G_{h/2}\ast\right]u \,dx $ (and the symmetry of $\sigma$).
  Note that due to the symmetry of $G$ there is a cancellation between the second and third term in this sum:
  \begin{align*}
    \int - \zeta \,u \cdot \sigma \, G_h \ast \left( \left(\xi \cdot \nabla\right) u\right) - \left( \xi \cdot \nabla\right) u \cdot \sigma  \left[ \zeta, G_h\ast\right] u\, dx
    = &\int -\zeta \left(\xi \cdot \nabla\right) u \cdot \sigma \, G_h\ast u\, dx\\
    =&\int - u \cdot \sigma \,G_h\ast \left(\zeta \left( \xi \cdot \nabla\right) u \right) dx.
  \end{align*}
  A direct computation based on the semi-group property $G_h = G_{h/2}\ast G_{h/2}$ yields
  \begin{equation}\label{commutator halbgruppe}
    \left[ \zeta, G_h\ast\right]v
    +\left[ \zeta, G_{h/2}\ast\right]G_{h/2}\ast v 
    - G_{h/2}\ast \left[ \zeta, G_{h/2}\ast\right]v 
    = 2 \left[ \zeta, G_{h/2}\ast\right]G_{h/2}\ast v
  \end{equation}
  so that the last three terms in the first variation of $E_h$ above can be combined using once more the antisymmetry of the commutator, and we get
  \begin{align}
    \frac{d}{ds}\Big|_{s=0}E_h (u_s,\chi;\zeta) = 
    &\frac1\h\int - \zeta \left(\xi \cdot \nabla \right)u\cdot \sigma \, G_h\ast u - u\cdot \sigma \,G_h\ast \left(\zeta\left( \xi \cdot \nabla \right) u \right)  dx \notag\\
    &-\frac2\h \int G_{h/2}\ast\left(u-\chi\right)\cdot \sigma \, \left[\zeta,G_{h/2}\ast\right] \left(\left(\xi \cdot \nabla \right)u  \right)dx.\label{d ds E}
  \end{align}
  Note that the first right-hand side integral is exactly $\delta E_h(u,\zeta\,\xi)$, the first variation of the energy along the ``localized'' vector field $\zeta\,\xi$, see  \eqref{dE without any tricks}.
  Now we plug $u=u^h$ into the above formula.
  Since by Lemma \ref{la conv ass for var int} $u^h\to \chi$ in $L^1$ and $E_h(u^h) \to E(\chi)$ for a.e.\ $t$,  using Proposition \ref{prop dE and dE old} with $\zeta\,\xi$ playing the role of $\xi$,  for a.e.\ $t$, along the sequence $u^h$ 
  the first right-hand side integral of \eqref{d ds E} converges to
  \[
      \delta E(\chi,\zeta\,\xi) 
      = c_0 \sum_{i,j} \sigma_{ij} \int \nabla \left( \zeta\, \xi\right) 
    \colon \left(Id- \nu_i\otimes \nu_i\right) \dS.
  \]

  We now give the argument that the second integral in \eqref{d ds E} is negligible:
  \begin{equation}\label{dE error term}
    \frac2\h \int G_{h/2}\ast\left(u-\chi\right) \cdot \sigma \left[\zeta,G_{h/2}\ast\right] \left( \left(\xi \cdot \nabla\right) u  \right)dx\to 0\quad \text{in } L^1(0,T).
  \end{equation}
   In view of \eqref{quantitative closeness} in Corollary \ref{cor a priori estimate}, by Cauchy-Schwarz it is enough to prove
  \begin{equation}\label{neu}
   \h \int_0^T\int \left| \left[\zeta,G_{h/2}\ast\right]\left(\xi \cdot \nabla u_i^h\right)\right|^2dx\,dt \to0
  \end{equation}
  for all $i=1,\dots,P$. We fix $i$ and omit the index in the following.
  Rewriting the commutator
  \begin{align*}
	\left[\zeta,G_{h/2}\ast\right]\left(\xi \cdot \nabla u^h\right)
   =\int  G_{h/2}(z) \left( \zeta(x)-\zeta(x-z)\right) \xi(x-z)\cdot \nabla u^h(x-z)\,dz
  \end{align*}
 and integrating by parts in $z$ we obtain the pointwise estimate
  \begin{align*}
   \left|\left[\zeta,G_{h/2}\ast\right]\left(\xi \cdot \nabla u^h\right)\right|
   \leq& \left|\int\nabla  G_{h/2}(z) \cdot \xi(x-z) \left( \zeta(x)-\zeta(x-z)\right) u^h(x-z) \,dz\right|\\
   &+ \left|\int G_{h/2}(z) \nabla_z \cdot \left[( \zeta(x)-\zeta(x-z))\xi(x-z)\right] u^h(x-z)\,dz \right|\\
   \lesssim & \|\nabla \zeta\|_\infty \|\xi\|_\infty+ \| \zeta\|_\infty \| \nabla \xi\|_\infty
  \end{align*}
  and hence \eqref{neu} holds with the rate $O\big((\|\nabla \zeta\|_{\infty} \|\xi\|_{\infty}
  +\| \zeta\|_{\infty} \|\nabla \xi\|_{\infty})^2 T \h\big)$.
  Therefore we have proven the following convergence of the first variation of the localized energy \eqref{def local E}:
  \begin{align}
    \lim_{h\to0}&\int_0^T \frac{d}{ds}\Big|_{s=0} E_h(u^h_s,\chi^h;\zeta) \,dt 
    = \lim_{h\to0} \int_0^T\delta E_h(u^h,\zeta\,\xi)\,dt \notag\\
    &= c_0\sum_{i,j} \sigma_{ij} \int_0^T\int \nabla\left(\zeta\,\xi\right) \colon \left(Id-\nu_i\otimes \nu_i\right) \dS dt.\label{limit of dE}
  \end{align}
  
  With the same methods we can handle the term in the expansion of the metric term $ \lauxbrakkedist_h(u^h_s,u^h;\zeta)$:
  We claim that
  \begin{align}
    \lim_{h\to0} 2 \h \int_0^T \int \zeta & \left|G_{h/2}\ast\left((\xi \cdot \nabla) u^h\right)\right|_\sigma^2 dx\,dt\notag\\  
   & = \lim_{h\to0}\frac2\h \sum_{i,j} \sigma_{ij} \int_0^T \int \zeta \left( \xi \otimes \xi\right) \colon u_i^h( h\nabla^2 G_h) \ast u_j^h\, dx \,dt \notag \\
    &= 2 c_0 \sum_{i,j}\sigma_{ij} \int_0^T  \int \zeta \left(\xi \cdot \nu_i \right)^2 \dS dt.\label{limit ddist}
  \end{align}
  To this end we plug $\left(\xi \cdot \nabla\right) u^h = \nabla \cdot (\xi u^h) - \left(\nabla \cdot \xi\right) u^h$ into the quadratic term on left-hand side
  and expand the square. First we note that only the term
  \begin{equation}\label{leading order}
    2 \h \int \zeta \left| G_{h/2}\ast\left(\nabla \cdot(\xi  u^h)\right)\right|_\sigma^2 dx = 2 \h \int \zeta \left| \nabla G_{h/2}\ast\left( \xi  \otimes u^h\right)\right|_\sigma^2 dx
  \end{equation}
  survives in the limit $h\to0$. Indeed, we have
  \[
     2 \h \int \zeta \left|G_{h/2}\ast\left((\nabla \cdot \xi)\, u^h\right)\right|_\sigma^2 dx \lesssim \|\nabla \xi\|_\infty^2 \h \int |\zeta|\,dx
  \]
  and the mixed term can be estimated by Young's inequality and the boundedness of the leading-order term which we will show now.
  Using the antisymmetry of $\nabla G$ we have
  \begin{align*}
   2\h \int \zeta  \big|\nabla  G_{h/2}\ast \big(\xi \otimes u^h \big) \big|_\sigma^2dx 
  & = 2\h \int u^h\cdot \sigma \left(\xi \cdot \nabla\right) G_{h/2}\ast  \big(\zeta\,\nabla  G_{h/2}\ast \big(\xi\, u^h \big) \big)\,dx.
  \end{align*}
  We now want to commute the multiplication with $\xi$ and the outer convolution and afterwards the multiplication with $\zeta\,\xi$ and the inner convolution. For this we use the $L^\infty$-commutator estimate
   \[
   \left\| \left[\xi, \nabla G_h\ast\right]u \right\|_\infty  \lesssim \|\nabla \xi \|_\infty \|u\|_\infty
   \]
   for the vector fields $\xi$ and $\zeta \xi$, which implies
    the $L^1$-estimate 
    \[
    \int \left| \nabla G_{h/2}\ast \left( \xi \, u_j^h \right) \right|dx \lesssim \|\nabla \xi \|_\infty
    + \|\xi\|_\infty \int \left| \nabla G_{h/2}\ast u_j^h\right|dx,
    \]
    and the a priori estimate \eqref{energy dissipation estimate interpolation} for the last term:
    \begin{equation}\label{nabla G energy}
     \limsup_{h\downarrow0} \int \left| \nabla G_{h/2}\ast u_j^h\right|dx 
    \lesssim \limsup_{h\downarrow0} E_h(u^h) 
    \stackrel{\eqref{energy dissipation estimate interpolation}}{<} \infty.
    \end{equation}
    For the first estimate in \eqref{nabla G energy} we exploited $\int \nabla G(z)\,dz=0$ as follows
    \begin{align*}
     \int \left| \nabla G_{h/2}\ast u_j^h\right|dx  
     &=\int \bigg| \int \nabla G_{h/2}(z) \left( u_j^h(x)-u_j^h (x-z)\right)dz \bigg|dx\\
    & \leq \iint |\nabla G_{h/2}(z)| \left| u_j^h(x)-u_j^h (x-z)\right| dz \,dx
    \end{align*}
    and used the pointwise estimates $|\nabla G_{h/2}(z)| \lesssim \frac1\h G_h(z)$ and $|u-v| \leq (1-u)v+u(1-v)$ for any $u,v\in[0,1]$. This gives indeed
    \[
     \int \left| \nabla G_{h/2}\ast u_j^h\right|dx   \lesssim  \frac2\h \int (1-u_j^h) \, G_h\ast u_j^h\,dx 
     \leq \frac1{\min_{i\neq j} \sigma_{ij}} E_h(u^h).
    \]
    
    Therefore, the left-hand side of \eqref{limit ddist} is indeed to leading order given by
  \[
   \frac2\h  \int \zeta \left( \xi \otimes \xi\right) \colon u^h \cdot \sigma ( h\nabla^2 G_h) \ast u^h\, dx.
  \]
  Then \eqref{limit ddist} follows from the convergence of the energies (cf.\ Lemma \ref{la conv ass for var int})
  and Proposition \ref{prop conv G K}.
  
  Using \eqref{limit of dE} for the numerator  and \eqref{limit ddist} for the denominator of the right-hand side of \eqref{lower bound dE} and $(\xi \cdot \nu_i)^2 \leq |\xi|^2$ along the way, we obtain by Fatou's Lemma in $t$
  \begin{align*}
    \liminf_{h\to 0}&\int_0^T |\partial E_h(\,\cdot\,,\chi^h;\zeta)|^2(u^h)\, dt\\
    &\geq \frac{c_0}2 \int_0^T \left( \sup_\xi \frac{ \int \nabla \left(\zeta\,  \xi \right) \colon\left( Id - \nu_i \otimes \nu_i\right) \sum_{i,j} \sigma_{ij}  \dS }
    {\sqrt{ \int \zeta\, |\xi|^2   \sum_{i,j} \sigma_{ij}  \dS}}
    \right)^2 dt.
  \end{align*}
  Applying this estimate for $u^h(h)=\chi^h(t+h)$ and $\zeta=1$ furnishes the existence of the mean curvature vector
  \[H\in L^2\left(\sum_{i,j}\sigma_{ij}\dS dt, \R^d\right)\]
  as claimed in the proposition. 
  Turning back to the interpolation $u^h$, we obtain the desired $\liminf$ inequality \eqref{H and dE} for the interpolations as well.
  \end{proof} 
 
%
  \bibliographystyle{plain}

\bibliography{lit}
  
 \end{document}